\documentclass[%
 reprint,
superscriptaddress,
 amsmath,amssymb,
 aps,
]{revtex4-2}

\usepackage{graphicx}% Include figure files
\usepackage{dcolumn}% Align table columns on decimal point
\usepackage{bm}% bold math
\usepackage{algorithm2e}

\usepackage{todonotes}

\usepackage{hyperref}% add hypertext capabilities
\usepackage{amsthm}
\usepackage{comment}
\usepackage{enumitem}
\usepackage{mathtools}
\usepackage{natbib}
\usepackage{amsfonts}
\usepackage{amssymb}

\newtheorem{prop}{Proposition}
\newtheorem{claim}{Claim}

\newtheorem{rem}{Remark}
\newtheorem{defi}{Definition}
\newtheorem{lemma}{Lemma}
\newtheorem{theo}{Theorem}

\newcommand{\sgn}{\text{sgn}}

\newcommand{\smin}{\setminus}

\usepackage {tikz}
\usetikzlibrary {positioning}
\tikzset{main node/.style={circle,fill=blue!40,draw,minimum size=1.2cm,inner sep=0pt},}
\tikzset{main node two/.style={circle,fill=blue!40,draw,minimum size=1.5cm,inner sep=0pt},}
\tikzset{small dot/.style={fill=black,circle,scale=1.5}}
\tikzset{small dot two/.style={fill=red,circle,scale=1.5}}

\usepackage{mathabx}

\usepackage{subcaption}
\usepackage{bbm}

\usepackage{lmodern}
\usepackage{physics}

\usetikzlibrary{fit}

\begin{document}

% \title{Mathematical vs. physical representations of quantum resource theories and general partial orders}

\title{The infinite information gap between mathematical and physical representations}

 %\title{Optimization principles or measurement outcomes?} 

\author{Pedro Hack}
\email{pedro.hack@dlr.de}
\affiliation{Ulm University, Germany}
\affiliation{Technical University of Munich, Germany}
\affiliation{German Aerospace Center, Germany}

\author{Daniel A.~Braun}
\email{daniel.braun@uni-ulm.de}
\affiliation{Ulm University, Germany}

\author{Sebastian Gottwald}
\email{sebastian.gottwald@uni-ulm.de}
\affiliation{Ulm University, Germany}

\begin{abstract}
Partial orders have been used to model several experimental setups, going from classical thermodynamics and general relativity to the quantum realm with its resource theories. In order to study such experimental setups, one typically characterizes them via a (numerical) representation, that is, a set of real-valued functions.
In the context of resource theory, it is customary to use \textbf{mathematical} representations, i.e. a set of \textbf{measurement outcomes} which characterize the achievable transitions within the experimental setup.
However, in line with the minimum energy and maximum entropy principles in classical mechanics and thermodynamics, respectively, one would expect an optimization interpretation for a representation to be called \textbf{physical}. More specifically, a physical representation could consist of a set of competing \textbf{optimization principles} such that a transition happens provided they are all optimized by it. Somewhat surprisingly, we show that this distinction can result in an \textbf{infinite information gap}, with some partial orders having mathematical representations that involve a finite amount of information and requiring infinite information to build a physical representation. We connect this phenomenon with well-known resource-theoretic scenarios like majorization, and develop notions of  partial order dimension that run in parallel to the representations that we consider. Our results improve on the classification of preordered spaces in terms of real-valued functions.
\end{abstract}

\providecommand{\keywords}[1]
{
  \small	
  \textbf{\textit{Keywords---}} #1
}


\maketitle

\section{Introduction}

Assume $(X,\preceq)$ is some experimental context \cite{sorkin1991spacetime,lieb1999physics,giles2016mathematical,chitambar2019quantum,surya2019causal}, where the \textbf{base} set $X$ is the set of possible states that our system can take and $\preceq$ determines an \textbf{ordering}, that is, which transition between pairs of elements $x,y \in X$ we can perform within this experimental setup, with $x \preceq y$ denoting that we can indeed transition from $x$ to $y$. A simple example of such an experimental setup is the \textbf{2D gravity ordering}, that is, the free fall of objects with null initial velocity under gravity $(\mathbb R^2,\preceq_g)$, where $X \equiv \mathbb R^2$ and a transition from $(x,y) \in \mathbb{R}^2$ to $(z,t) \in \mathbb{R}^2$ can be achieved provided they have the same horizontal component $x=z$ and the height of the second state is smaller than that of the first state $y\geq t$, that is,
\begin{equation*}
(x,y) \preceq_g (z,t)  \iff
    x = z \text{ and } y \geq t.
\end{equation*}

A system can have \textbf{reversible} transitions, in the sense that both $x \preceq y$ and $y \preceq x$ hold, which we denote by $x \sim y$. In order to make our setup interesting, we assume our system also has \textbf{irreversible} transitions, in the sense that there exist pairs $x,y \in X$ for which $x \preceq y$ holds while $y \preceq x$ does not, which we denote by $x \prec y$ \footnote{Irreversibility is key in order to determine the $X$ we consider. We could take $X$ to be the set of physical systems in the universe $U$. However, this would mean several of them would only have \textbf{trivial transitions} $\{y \in X | x \preceq y\} = \{x\}$, since they do not interact in a meaningful way with our experimental setup. Hence, we could establish as a criterion for $x \in U$ to be in $X$ that there must be some irreversible transition either starting or ending at $x$.}. In the 2D gravity ordering, any state can only reversible transition to itself, and the irreversible transitions are those that that leave the horizontal coordinate fixed and reduce the vertical coordinate. 

It is natural to assume that our setup $(X,\preceq)$ is a \textbf{partial order} \cite{bridges2013representations}, that is, that it fulfills the following properties:
\begin{itemize}
    \item \textbf{Reflexivity}: One can always transition from any state to itself $x \preceq x$.
    \item \textbf{Transitivity}: One can concatenate two transitions, that is, if $x \preceq y$ and $y \preceq z$, then $x \preceq z$.
    \item \textbf{Antisymmetry}: If one can transition from a state to another, $x \preceq y$, and vice versa, $y \preceq x$, then they are the same state $x=y$. We assume this for simplicity. It will become clear that, regarding the properties we are interested in, this assumption does not make any difference, i.e. our results hold for preorders as well.
\end{itemize}

Partial orders are typically represented via a \textbf{Hasse diagram} \cite{harzheim2006ordered}, where we include a finite set of points that correspond to some subset of the states $H \subseteq X$ and a line joining a lower point $x$ to an upper point $y$ is included provided $x \prec y$ \footnote{Hasse diagrams are typically defined for finite $X$ and, in these cases, one takes $H \equiv X$ and has \textbf{the} Hasse diagram of the partial order. However, we are going to consider here $X$ to be infinite almost exclusively. Hence, we talk about \textbf{a} Hasse diagram.}. To avoid having too many lines, Hasse diagrams do not add those that can be constructed by concatenating other lines. We provide a Hasse diagram for the 2D gravity ordering in Figure \ref{fig:gravity}.  

 \begin{figure}[!tb]
\centering
\begin{tikzpicture}
    \node[main node] (1) {$(x,h_0)$};
    \node[main node] (2) [right = 1.5cm  of 1]  {$(y,h_0)$};
    \node[main node] (3) [right = 1.5cm  of 2]  {$(z,h_0)$};
    \node[main node] (4) [below = 2cm  of 1] {$(x,h_1)$};
    \node[main node] (5) [right = 1.5cm  of 4] {$(y,h_1)$};
    \node[main node] (6) [right = 1.5cm  of 5] {$(z,h_1)$};

    \path[draw,thick,-]
    (4) edge node {} (1)
    (6) edge node {} (3)
    (5) edge node {} (2)
    ;
\end{tikzpicture}
\caption{A Hasse diagram of the 2D gravity ordering. We assume here $x<y<z$ and $h_0<h_1$.}
\label{fig:gravity}
\end{figure}

It is important to note that our experimental setup does not include any reference to how one may produce a certain state of the system in order to start some experiment. In fact, we assume some of the states of the system are relatively complicated to produce and hence are considered \textbf{resources} \cite{gour2015resource,winter2016operational,chitambar2019quantum}. Since the transitions $\preceq$ are actually at our disposal, we think of them as being \textbf{free operations}, i.e. operations that are relatively simple to perform compared to the production of resources. Now, if we intend to actually use our experimental tools, plus some more complicated protocol that produces resources, it is key to understand what states one can obtain via $\preceq$ provided one is given a certain resource. In the 2D gravity ordering, having an object at a certain height is a resource, since we can use its potential energy to produce work. When left at that height, we do not need to do any work for the object to fall and, hence, we can think of falling as a free operation. To give an example coming from the so-called \textbf{quantum resource theories}, we can consider, in the context of obtaining non-Clifford gates, \textbf{magic states} and Clifford gates as resources and free operations, respectively \cite{howard2017application}.

Before we continue, let us consider two questions:
\begin{itemize}
    \item \textbf{What is a state?} Abstractly, we can consider it as a series of interactions between our experimental devices and the state or \textbf{measurement outcomes}. That is, whatever the state $x \in X$ is, for us it is simply a set of values that determine that we have $x$ at our disposal, and not any other $y \in X$ ($y \neq x$), and what we want to determine is the set of measurement outcomes we can achieve from some initial measurement outcomes. 
    \item \textbf{What kind of properties are we measuring?} These properties should be meaningful for the experimental setup at our disposal, that is, the values that they take relative to each other should not be random but correlated to whether transitions between them are achievable. In this sense, it is natural to consider \textbf{monotones}, i.e. functions $u:X \to \mathbb R$ such that, if $x \preceq y$, then $u(x) \leq u(y)$. In the 2D gravity ordering, \textbf{negative gravitational potential energy} $E_g((x,y)) \equiv -mgy$ is a monotone. 
\end{itemize}

Going back to the numerical determination of the possible transitions, we can use a set of monotones to characterize them. In this regard, it would be enough to have a set of them $(u_i)_i$ such that, if a transition is not possible, then there would be some $u_{i_0} \in (u_i)_i$ that decreases (instead of increasing). In the 2D gravity ordering, a simple set of functions that achieves this is the negative gravitational potential together with the function that associates to a point its horizontal coordinate $\mathcal X ((x,y)) \equiv x$ and one that associates to it the negative of its horizontal coordinate $\mathcal X_{\text{neg}} ((x,y)) \equiv -\mathcal X ((x,y))$, which together allow to determine whether $x=z$ or not:
\begin{equation}
\label{eq: gravity mu}
(x,y) \preceq_g (z,t)  \iff
\begin{cases}
    E_g((x,y)) \leq E_g((z,t)),\\
    \mathcal X ((x,y)) \leq \mathcal X ((z,t)) \text{, and}\\
    \mathcal X_{\text{neg}} ((x,y)) \leq \mathcal X_{\text{neg}} ((z,t)).
    \end{cases}
\end{equation}

While this picture may be satisfactory for instance in the practical context we described before, we are interested here in how this picture is related to more strict sorts of numerical representations that are closer in spirit to the classical interpretations of physical phenomena as \textbf{optimization processes}, like
\begin{itemize}
    \item the \textbf{minimum energy principle} in classical mechanics,
    \item or the \textbf{maximum entropy principle} in classical thermodynamics.
\end{itemize}

\subsection{Contribution}

Our main contributions are the following:
\begin{itemize}
    \item We make a fundamental distinction between numerical representations of experimental setups (Section \ref{sec: math vs phys repre}), depending on whether we require them to be interpretable as optimization principles (\textbf{physical representations}) or we simply require them to characterize the transitions achievable within our setup (\textbf{mathematical representations}).

    \item We show the existence of an infinite \textbf{information gap} (Section \ref{sec: fin math and infinite phys}), that is, we show there are experimental setups that require infinite information to be described as optimization processes via a physical representation, while finite information suffices in the non-optimization case concerning mathematical representations \footnote{We provide partial orders where, despite the existence of finite mathematical representations, finite physical representations do not exist. This should not be confused with previous work on the existence of resource theories without finite mathematical representations \cite{datta2023there}.}. 
    \item We compute the information gap for (finite and infinite) \textbf{majorization} (Section \ref{sec: fin + inf majo}), complementing \cite[Theorem 1]{hack2022disorder}.
    \item We provide an 
    easy-to-check condition under which the information gap vanishes (Section \ref{sec: no info gap}).
    \item We relate the different numerical representations we have considered with notions of \textbf{dimension} for partial orders (Section \ref{sec: order dim}),  introducing two new notions of dimension and showing the existence of an infinite information gap for them.
    \item Our findings also improve on the classification or preordered spaces in terms of real-valued monotones \cite{hack2022representing,hack2022classification}. We visualize our contributions in Figure \ref{fig classi}.
\end{itemize}

\section{Mathematical vs. physical representations}
\label{sec: math vs phys repre}

A \textbf{numerical representation} or \textbf{representation} for simplicity is a set $(u_i)_{i \in I}$ of real-valued functions  $u_i: X \to \mathbb R$ that is used in order to study a transition system $(X,\preceq)$.

We distinguish two types of representations:
\begin{itemize}
    \item In a \textbf{mathematical representation}, we can think of $(u_i)_{i \in I}$ as a set of \textbf{measurement instruments}, in the sense that, for any pair $x,y \in X$, we can determine whether there exists a process within some specific experimental context $\preceq$ connecting $x$ and $y$ by comparing the tuples of \textbf{measurement outcomes} 
    \begin{equation*}
     \left\{(u_i(a),u_i(b))\right\}_{i \in I}.   
    \end{equation*}
     Hence, a mathematical representation has an operational value, in the sense that we expect to identify the set of realizable transitions by using it.
    %This goes in the direction of the Copenhagen interpretation.
    \item In a \textbf{physical representation}, one would look for a set of optimization principles that explain the behaviour of the system, like the minimum energy principle in classical mechanics and the maximum entropy principle in classical thermodynamics. That is,
    we think of $(u_i)_{i \in I}$ as a set of competing optimization principles such that a transition from $x$ to $y$ happens provided all the optimization principles agree that it should happen. By this, in the spirit of thermodynamic entropy \cite{landau2013statistical}, we mean the following:
    \begin{itemize}
        \item If a transition is \textbf{reversible}, then all optimization functions should remain unchanged.
        \item If a transition is \textbf{irreversible}, then all optimization functions should increase.
        \item If a transition is \textbf{impossible}, then at least one optimization function should decrease.
    \end{itemize}
\end{itemize}

 It should be emphasized that physical representations are also mathematical, since we can use the value of the optimization functions as measurement outcomes and determine the achievable transitions from them. However, we do not require every mathematical representation to be physical \footnote{In fact, we will show that some of them are not physical in Section \ref{sec: fin math and infinite phys}.}.

Examples of mathematical \cite{ok2002utility,evren2011multi,goold2016role,gour2018quantum,sagawa2022entropy} and physical \cite{klimesh2004entropy,turgut2007catalytic,brandao2015second,muller2016generalization,alcantud2016richter,hack2022representing} representations have been considered in both the study of abstract partial orders and quantum resource theories. We focus here on the most prominent ones:

\begin{itemize}
    \item A representation $(u_i)_{i \in I}$ is a \textbf{multi-utility} \cite{evren2011multi} if we have, for all $x,y \in X$, that $x \preceq y \iff u_i(x) \leq u_i(y)$ for all $ i \in I$. Multi-utilities are mathematical representations.
    \item A representation $(u_i)_{i \in I}$ is a \textbf{strict monotone multi-utility} \cite{alcantud2016richter} if it is a multi-utility and, for each $i \in I$, we have that 
    $x \prec y$ implies $u_i(x)< u_i(y)$ for all $ x,y \in X$, where $x \prec y$ stands for $x \preceq y$ and $\neg(y \preceq x)$ \cite{richter1966revealed,peleg1970utility}. Strict monotone multi-utilities are physical representations. This is the case since the extra requirement of the members of the representation assures that maximizing over them is equivalent to maximizing over the partial order itself, that is, they are legitimate optimization principles. 
\end{itemize}

 In the following, we will refer to multi-utilities as mathematical representations and to strict monotone multi-utilities as physical representations. For the 2D gravity ordering, \eqref{eq: gravity mu} is a mathematical representation which is not a physical representation. In fact, as we will show in Section \ref{sec: fin math and infinite phys}, the 2D gravity ordering only admits physical representations consisting of an infinite number of functions.

If we have a mathematical representation that is not a physical representation $(u_i)_{i \in I}$, then there will be some function $u_{i_0}$ for which we can find a region of the state space $R \subseteq X$ such that, when optimized over $R$, $u_{i_0}$ will output \textbf{non-equilibrium states}, i.e., non-optimal states. (See \cite{hack2022disorder} for more details.) Hence, we cannot interpret $(u_i)_{i \in I}$ as a set of competing optimization principles.

 Because of the optimization requirements, the functions within a physical representation are closer to the classical \textbf{thermodynamical entropy}, which has always been considered as an optimization principle \cite{jaynes1957information,lieb1999physics,landau2013statistical}. In fact, in the accompanying paper \cite{hack2022disorder}, we argue how, taking majorization as the \textbf{arrow of time}, physical representations are more legitimate generalizations of the \textbf{second law} of thermodynamics than mathematical representations.

If we strengthen the requirements on entropy \cite{jaynes1957information}, and we require it to provide an \textbf{unambiguous equilibrium state} when optimized over some set (assuming such an optimum exists), we end up with a second sort of physical representation:

\begin{itemize}
    \item A representation $(u_i)_{i \in I}$ is an \textbf{injective monotone multi-utility} if it is a multi-utility and, for each $i \in I$, we have that 
    $u_i(x)=u_i(y)$ implies $x = y$ for all $ x,y \in X$ \cite{hack2022representing}. The injectivity property ensures that the representation is a strict monotone multi-utility and, moreover, it also enforces the uniqueness of equilibrium states.
\end{itemize}

In the following, we will refer to injective monotone multi-utilities as injective physical representations. Injective monotone multi-utilities will only appear later on when connecting numerical representations with order dimension.

Given some partial order, a key question is:
\begin{itemize}
    \item What is the \textbf{minimal} number of functions we can have in a specific sort of representation?
\end{itemize}
The main results of this work will address this question, showing somewhat surprising jumps in minimality when slightly changing the requirements on the representations. 

\subsection{Examples}

Before going further, let us give a few examples of representations:
\begin{itemize}
    \item Negative gravitational potential energy $E=-mgh$ constitutes a single-function mathematical and physical representation in the context of one-dimensional free fall in classical mechanics.  
    \item Entropy constitutes a single-function mathematical and physical representation in the context of classical thermodynamics \cite{jaynes1957information,landau2013statistical,lieb1999physics}.
    \item Majorization $(\mathbb P_\Omega^{\downarrow},\preceq_m)$ \cite{arnold2018majorization,nielsen1999conditions,nielsen2002introduction,lostaglio2019introductory} is a partial order defined on $\mathbb P_\Omega^{\downarrow}$, the set of probability distributions on a finite set $\Omega$ whose components are non-increasingly ordered. The ordering is defined as follows:
\begin{equation}
\label{majo}
\begin{split}
    &p \preceq_m q \ \iff \ u_i(p) \leq u_i(q) \text{ for } i=1,..,|\Omega|-1, \\
    &u_i(p) \equiv \sum_{n=1}^{i} p_n.
    \end{split}
\end{equation}
%where $u_i(p) \coloneqq \sum_{n=1}^{i} p_n$.
The definition of majorization \eqref{majo} is through a mathematical representation that is not physical (see \cite{hack2022disorder}). We will address the minimality of mathematical and physical representations in Section \ref{sec: majo}.
\end{itemize}

A non-physically-motivated well-know partial order is the power set under set inclusion:

\begin{itemize}
    \item Given some (for simplicity) finite set $A=\{1,\dots,n\}$, the \textbf{set-inclusion power set} $(2^A,\preceq_{2^A})$ is a partial order that has the power set $2^A$ as base set and set inclusion
\begin{equation*}
x \preceq_{2^A} y \iff x \subseteq y
\end{equation*}
 as order relation.
\end{itemize}

We show the Hasse diagram of a set inclusion power set with $A \equiv \{1,2\}$ in Figure \ref{fig:set-inclusion}, and discuss the minimal physical and mathematical representations of $(2^A,\preceq_{2^A})$ for general $A$ in Appendix \ref{sec: lemma power set}.

 \begin{figure}[!tb]
\centering
\begin{tikzpicture}
    \node[main node] at (0,-2) (1) {$\emptyset$};
    \node[main node] at (-2,0) (2)  {$\{1\}$};
    \node[main node] at (2,0) (3)  {$\{2\}$};
    \node[main node] at (0,2) (4) {$\{1,2\}$};

    \path[draw,thick,-]
    (1) edge node {} (2)
    (1) edge node {} (3)
    (2) edge node {} (4)
    (3) edge node {} (4)
    ;
\end{tikzpicture}
\caption{The Hasse diagram of the set-inclusion power set with $A \equiv \{1,2\}$.}
\label{fig:set-inclusion}
\end{figure}

\section{Finite mathematical and infinite physical representations}
\label{sec: fin math and infinite phys}

The difference between mathematical and physical representations does not seem to be large. It is for instance known \cite{alcantud2016richter} that, for any partial order, there exists a countable mathematical representation if and only if there exists a countable physical representation. However, as we show in the following theorem via the 2D gravity ordering, this result breaks at the finite level. 

\begin{theo}
\label{thm: gravity}
There exist partial orders with finite mathematical representations and countably infinite minimal physical representations. In fact, the 2D gravity ordering is such a partial order.
\end{theo}

We show Theorem \ref{thm: gravity} in Appendix \ref{sec: gravity}. Theorem \ref{thm: gravity} shows that one can have an experimental setup whose transitions can be predicted in a finite manner but cannot be interpreted as a series of competing optimization principles without an infinite amount of information. Importantly, we have shows that this happens in a very basic experimental setup (the 2D gravity ordering). We show the same happens for majorization in Section \ref{sec: majo}. 

The 2D gravity ordering has a minimal mathematical representation consisting of three functions (see Appendix \ref{sec: gravity}). Moreover, a mathematical representation consisting of one function is a physical representation as well. It remains then to establish a sharp lower bound on the mathematical representations for which Theorem \ref{thm: gravity} holds. We do so in the following theorem:

\begin{theo}
\label{finite geo infinite deb}
There exist partial orders with minimal mathematical representations consisting of two functions and countably infinite minimal physical representations. 
\end{theo}

  We show Theorem \ref{finite geo infinite deb} in Appendix \ref{sec: proof main theo}.

\section{Majorization}
\label{sec: fin + inf majo}

\subsection{Finite Majorization}
\label{sec: majo}

The construction in Theorem \ref{finite geo infinite deb} is abstract, but it actually manifests itself in well-known quantum resource theories, like majorization. In fact, by finding an \textbf{order isomorphism} \footnote{An order isomorphism between a pair of partial orders $(X,\preceq)$ and $(Y,\preceq')$ is a bijection $f:X \to Y$ such that $x \preceq y$ if and only if $f(x) \preceq' f(y)$
 \cite{ore1987theory}.} from a subset of majorization to the construction in Theorem \ref{finite geo infinite deb}, one can show the following \footnote{The use of the order isomorphism relies on the fact that, if $(X,\preceq)$ is a partial order and $(Y,\preceq')$ is the partial order that takes some $Y \subseteq X$ and together with the restriction of $\preceq$ to $Y$, $\preceq'$, then, provided $(Y,\preceq')$ does not have a physical representation consisting of $k$ function, then the same holds for $(X,\preceq)$.}:  

\begin{theo}
\label{dim majo}
If $|\Omega| \geq 3$, then majorization $(\mathbb P_\Omega^{\downarrow},\preceq_m)$ has a minimal mathematical representation consisting of $|\Omega|-1$ functions and a countably infinite minimal physical representation.
\end{theo}

%Theorem \ref{dim majo} has the same interpretation as Theorem \ref{finite geo infinite deb}, but now in a relevant physical context. 
Although the statement regarding the minimal mathematical representation of majorization is quite natural, we have not found a proof or even the statement in the literature.
The proof of Theorem \ref{dim majo} is split between Appendix \ref{sec: proof major thm} and \cite[Theorem 1]{hack2022disorder}. 

\subsection{Infinite Majorization}
\label{sec: infinite majo}

As Theorem \ref{dim majo} seems to suggest, the gap disappears when we consider infinite majorization.
\textbf{Infinite majorization} $(\mathbb P_\infty^{\downarrow},\preceq_m^{\infty})$ \cite{li2013neumann} is a partial order defined on 
\begin{equation*}
\mathbb P_\infty^{\downarrow} \equiv  \left\{(p_n)_{n \geq 1} | 0 \leq p_n \leq 1 \text{, } p_n \geq p_{n+1} \text{, } \sum_{n \geq 1} p_n =1 \right\},
\end{equation*}
the set of probability distributions on a countably infinite set whose components are non-increasingly ordered. The ordering is defined as expected:
\begin{equation}
\label{def: inf majo}
    p \preceq_m^{\infty} q \ \iff \ u_i(p) \leq u_i(q) \text{ for } i \geq 1.
\end{equation}

As we show in Appendix \ref{app: infinite majo}, we have the following statement regarding the representations of infinite majorization:

\begin{theo}
\label{teo: inf majo}
Infinite majorization $(\mathbb P_\infty^{\downarrow},\preceq_m^{\infty})$ has countably infinite minimal mathematical and physical representations. 
\end{theo}

\section{Bridging the gap}
\label{sec: no info gap}

We would like to able to tell, given a partial order with a finite mathematical representation, whether a finite physical representation exists as well provided some conditions are met. This can be achieved through a well-known order denseness property called Debreu separabilty. A partial order $(X,\preceq)$ is \textbf{Debreu separable} \cite{bridges2013representations} provided there exists some countable set $D \subseteq X$ such that, for all $x,y \in X$ such that $x \prec y$, there exists some $d \in D$ such that $x \preceq d \preceq y$. We refer to such a $D \subseteq X$ as a countable \textbf{Debreu dense} subset \cite{bridges2013representations}.

Debreu separability is actually a fundamental property of representations given that, if the partial order is total, then it has a representation consiting of a single function if and only if it is Debreu separable \cite[Theorem 1.4.8]{bridges2013representations}. Despite its key role, it seems like the following equivalence has not been established:

\begin{theo}
\label{classification}
If $(X,\preceq)$ is a Debreu separable partial order and $k \geq 1$, then $(X,\preceq)$ has a minimal mathematical representation consisting of $k$ functions if and only if it has a minimal physical representation consisting of $k$ functions.
\end{theo}

We prove Theorem \ref{classification} in Appendix \ref{sec: debreu thm}, where we also show that the existence of finite physical representations does not ensure Debreu separability.
Moreover, in the restricted case of minimal representations consisting of two functions ($k=2)$, we can extend the equivalence in Theorem \ref{classification} to injective physical representations. Theorem \ref{classification} does not apply to majorization in general, since majorization is not Debreu separable provided $|\Omega| \geq 3$ \cite[Lemma 5]{hack2022representing}.

Theorem \ref{classification} together with Theorem \ref{finite geo infinite deb} and Proposition \ref{finite RP no deb sep} are improvements on the classification or preordered spaces in terms of real-valued monotones \cite{hack2022representing,hack2022classification}. Figure \ref{fig classi} contains a Venn diagram with these contributions. A more exhaustive diagram can be found in \cite{mythesis}. 

\begin{figure}[!tb]
\centering
\begin{tikzpicture}[scale=0.8, every node/.style={transform shape}]
\node[ draw,fill=blue!10,yshift= 0.375cm,text height = 10.75cm, minimum width = 10cm,label={[anchor=south,above=1.5mm]270:\textbf{Preordered spaces}}] (main) {};
\node[ draw,fill=blue!50, text height = 5cm, minimum width = 6cm,yshift=-1cm,xshift=1cm,label={[anchor=south,above=0.1mm]270: \textbf{Debreu separable}}] at (main.center) (non) {};
\node[ draw,fill=blue!30, text height = 6cm, minimum width = 6cm,yshift=0.7cm,xshift=-0.5cm,label={[anchor=north,below=1.5mm]90:\textbf{Finite multi-utility}}] at (main.center) (non) {};
\node[ draw,fill=blue!50, text height = 5cm, minimum width = 5cm,yshift=0.2cm,xshift=0cm,label={[anchor=north,below=1.5mm]90:\textbf{Finite strict}}] at (main.center) (non) {};
\node[ draw, text height = 5cm, minimum width = 5cm,yshift=0.2cm,xshift=0cm,label={[anchor=north,below=4.5mm]90:\textbf{monotone multi-utility}}] at (main.center) (non) {};
\node[ draw,fill=blue!70, text height = 2cm, minimum width = 2.8cm,yshift=-0.8cm,label={[anchor=north,below=1cm]90:\textbf{Utility function}}] at (main.center) (non) {};
\node[ draw, text height = 5cm, minimum width = 6cm,yshift=-1cm,xshift=1cm] at (main.center) (non) {};
%\node[ draw, text height = 4cm, minimum width = 4cm,yshift=-1cm,xshift=2cm] at (main.center) (non) {};
\end{tikzpicture}
\caption{Subset of the classification of preordered spaces in terms of real-valued monotones.}
%Our contribution to the classification includes Theorems \ref{finite geo infinite deb} and \ref{classification}, and Proposition \ref{finite RP no deb sep}.}
\label{fig classi}
\end{figure}

\section{Order dimension and representations}
\label{sec: order dim}

We can think of the minimal representations of each sort that we are considering as \textbf{measures of complexity} or \textbf{dimension}, in the sense that they tell us the minimal number of numerical functions which are needed in order for a characterization with certain properties to exist. Moreover, all the characterizations we have considered are actually tied to geometry, since each of these functions uses the real line as output space. While this connection to the real numbers is natural in physics in order to interpret these functions as measurement outcomes or optimization principles, the mathematical study of partial order complexity does not necessarily follow this restriction. In this section, we bridge the gap between the mathematical and physical approaches by introducing two new notions of dimension: The Debreu and geometrical dimensions. 

The standard definition of dimension for partial orders is due to Dushnik and Miller \cite{dushnik1941partially}, and it requires some concepts to be introduced first.
If $(X,\preceq)$ is a partial order, we say the partial order $(X,\preceq')$ is an \textbf{extension} of $(X,\preceq)$ if $x \preceq y$ implies $x \preceq' y$ for all $ x,y \in X$ \cite{harzheim2006ordered}.
%We may denote this by $\preceq \subseteq \preceq'$.
Moreover, $(X,\preceq')$ is a \textbf{linear extension} of $(X,\preceq)$ if it is an extension and a total partial order \cite{harzheim2006ordered}. A family of linear extensions $(\preceq_i)_{i \in I}$ is called a \textbf{realizer} of $\preceq$ if
\begin{equation*}
    x \preceq y \iff x \preceq_i y \text{ } \forall i \in I \text{ and } x,y \in X.
\end{equation*}
Finally, the \textbf{Dushnik-Miller dimension} of $(X,\preceq)$ is the cardinality of the smallest realizer $(\preceq_i)_{i \in I}$ of $(X,\preceq)$. 

The Ore dimension \cite{ore1987theory,rival2012ordered} is an alternative definition of the Dushnik-Miller which can be shown to be equivalent to it \cite[Theorem 10.4.2]{ore1987theory}. The \textbf{Ore dimension} of  $(X,\preceq)$ is the cardinality of the smallest set $I$
 such that $(X,\preceq)$
 is order-isomorphic to
 $(S,\preceq')$, where $S \subseteq \bigtimes_{i \in I} C_i$ is a subset of the Cartesian product of $I$ chains $(C_i)_{i \in I}$ and $\preceq'$ is the product-induced order with respect to the partial orders $(\preceq_i)_{i \in I}$ of the chains \footnote{The \textbf{product-induced} order $\preceq'$ with respect to a family of partial orders $((X_i,\preceq_i))_{i\in I}$ is the partial defined on the product space $\bigtimes_{i\in I} X_i$ given by $x\preceq' y$ if and only if $x_i \preceq_i y_i$ for all $ i \in I$.}.

Despite the Ore dimension being equivalent to the Dushnik-Miller dimension, it brings the concept closer to mathematical representations. In particular, the minimal physical representation or minimal multi-utility is precisely the Ore dimension if we restrict ourselves to the case $C_i \equiv \mathbb R$ and we take the standard ordering of the real line $\preceq_i \equiv \leq$ for all $i \in I$. Hence, we can think of this as a geometrical version of the Dushnik-Miller dimension, which motivates the following definition: 

\begin{defi}[Geometrical dimension]
\label{geo dim}
If $(X,\preceq)$ is a partial order, then its geometrical dimension is the cardinality of its minimal mathematical representation or multi-utility.
\end{defi}

Although multi-utilities have a long tradition \cite{ok2002utility}, we believe their minimal cardinality was not given a name before. 

The geometrical dimension and the Dushnik-Miller dimensions can be far apart. In fact, as we show in Appendix \ref{sec: dm and geo}, the following holds:
\begin{theo}
\label{teo: dm and geo}
There exist partial orders with finite Dushnik-Miller dimension and uncountably infinite geometrical dimension.
\end{theo}

The gap between these two notions of dimension comes from the disconnection between the Dushnik-Miller dimension and the real line. To overcome this, we can again use the fundamental relation between Debreu separability and the real line to introduce a new notion of dimension: 
\begin{defi}[Debreu dimension]
\label{deb dim}
If $(X,\preceq)$ is a partial order, then its Debreu dimension is the cardinality of its minimal Debreu separable realizer $(\preceq_i)_{i \in I}$ \footnote{By \textbf{Debreu separable realizer} we simply mean a realizer $(\preceq_i)_{i \in I}$ where each partial order is Debreu separable.}.
\end{defi}

Using \cite[Theorem 1.4.8]{bridges2013representations}, we can directly show the tight connection between the Debreu dimension and minimal physical representations:
\begin{lemma}
\label{deb reali = inj mu}
If $(X,\preceq)$ is a partial order and $I$ is an arbitrary set, then $(X,\preceq)$ has a minimal injective physical representation consisting of $|I|$ functions if and only if its Debreu dimension is $|I|$.
\end{lemma}

Lemma \ref{deb reali = inj mu} allows us to translate the results from representations to order dimensions. In fact, by directly applying \cite[Proposition 6]{hack2022representing} and Theorem \ref{finite geo infinite deb}, we have the following:

\begin{theo}
A partial order has a countable geometrical dimension if and only if it has a countable Debreu dimension. However, there exist partial orders with finite geometrical dimension and countably infinite Debreu dimension.
\end{theo}

\section{Conclusion}

The basic point we are trying to convey here is a distinction between sorts of numerical representations of physical theories and experimental setups, namely mathematical and physical ones. We have argued that the typical numerical representations that have been studied, i.e. mathematical representations, do not correspond to those used in classical physical theories, which are always closely related to optimization. As we show, adding a connection to optimization may have dramatic consequences: The amount of information required to describe a system may become infinite!

Fundamentally, we ask ourselves what we expect from measurement apparatuses, with our main motivation being to gain insight regarding the concept of \textbf{thermodynamical entropy} and to better understand \textbf{generalized entropies} and \textbf{nonextensive statistical mechanics} \cite{ruch1975diagram,ruch1976principle,mead1977mixing,tsallis2009introduction,campion2016entropy}.
 
We conclude by pointing out a couple of questions for further research:
It remains to determine what the minimal representations for \textbf{trumping} or \textbf{catalytic majorization} \cite{daftuar2001mathematical,turgut2007catalytic,elkouss2025finite} are. In this regard, one can use the tools developed here in order to show that the minimal physical representation are infinite. (See \cite[Corollary 2]{hack2022disorder}.) Moreover, it is still unclear how finite injective physical representations are related to non-injective finite physical representations. Establishing a connection between them would be very significant since injective physical representations correspond to the most demanding requirement on entropy funcitons that have been considered before \cite{jaynes1957information}. This would also complete the classification of preordered spaces in terms of real-valued monotones.

\section*{Acknowledgements}
We thank Gianni Bosi and Andreas Winter for useful feedback.
This research is funded by the European Research Council and the Munich Quantum Valley.

\bibliography{main}

\clearpage

\begin{appendix}

\section{Minimal representations of the set-inclusion power set}
\label{sec: lemma power set}

 We devote this section to show the following: 

\begin{itemize}
    \item If $A \equiv \{1,\dots,n\}$, then the minimal mathematical and physical representations of $(2^A,\preceq_{2^A})$ consist of $n$ functions.
\end{itemize}

We begin by showing the results for mathematical representations. The canonical mathematical representation \cite{evren2011multi} would be $(u_x)_{x \in 2^A}$, where
    \begin{equation*}
    u_x(y) \equiv
    \begin{cases}
    1 \text{ if } x \preceq_{2^A} y,\\
    0 \text{ otherwise,}
    \end{cases}
\end{equation*}
However, $(u'_x)_{x \in A}$ with
  \begin{equation*}
    u'_x(y) \equiv
    \begin{cases}
    1 \text{ if } x \in y,\\
    0 \text{ otherwise,}
    \end{cases}
\end{equation*}
is also a multi-utility and it consists of fewer functions. In order to show that there cannot be mathematical representations with less than $n$ functions, it is sufficient to find some $(B,\preceq_B) \subseteq (2^A,\preceq_{2^A})$, where $\preceq_B$ is the restriction of $\preceq_{2^A}$ to $B \subseteq 2^A$, that is order isomorphic to the partial order $(S_0,\preceq_{S_0})$, where 
\begin{equation*}
    S_0 \equiv \{m \in \mathbb Z | m \text{ or } -m \in \{1,\dots,n\} \} %\{-n,\dots,-1,1,\dots,n\}
\end{equation*}
and $\preceq_{S_0}$ is restriction of the partial order in \cite[Proposition 7]{hack2022classification} to $S_0$. Once this is done, we can follow the proof of \cite[Proposition 7]{hack2022classification} to get the desired result. The set 
\begin{equation*}
    B \equiv \{ \{i\} | i \in A \} \bigcup \{ A \smin \{i\} | i \in A \} 
\end{equation*}
has the desired properties. (We represent $(B,\preceq_B)$ when $n=3$ in Figure \ref{fig: power set}.)

To conclude, we address the result for physical representations. For simplicity, let us just note that, although simpler approaches are also possible, we can obtain a physical representation consist of $n$ functions by following the proof of Theorem \ref{classification}. It is then clear that such a representation is minimal since, if there were one with fewer functions, then it would also be a mathematical representation with fewer functions, and we would reach a contradiction. 

\begin{figure}[!tb]
\centering
\begin{tikzpicture}
%\node[rounded corners, draw,fill=blue!10, text height = 3cm, minimum width = 11cm,xshift=4cm,label={[anchor=west,left=.1cm]180:\textbf{B}}] {};
%\node[rounded corners, draw,fill=blue!10, text height = 3cm, minimum width = 11cm,xshift=4cm,yshift=-4cm,label={[anchor=west,left=.1cm]180:\textbf{A}}] {};
    \node[main node] (1) {$A \smin \{1\}$};
    \node[main node] (2) [right = 2cm  of 1]  {$A \smin \{2\}$};
    \node[main node] (3) [right = 2cm  of 2]  {$A \smin \{3\}$};
    \node[main node] (4) [below = 2cm  of 1] {$\{1\}$};
    \node[main node] (5) [right = 2cm  of 4] {$\{2\}$};
    \node[main node] (6) [right = 2cm  of 5] {$\{3\}$};

    \path[draw,thick,-]
    %(4) edge node {} (1)
    (4) edge node {} (2)
    (4) edge node {} (3)
    (5) edge node {} (3)
    %(5) edge node {} (2)
    (5) edge node {} (1)
    (6) edge node {} (1)
    (6) edge node {} (2)
    %(6) edge node {} (3)
    ;
\end{tikzpicture}
\caption{The Hasse diagram of $(B,\preceq_B)$ in Appendix \ref{sec: lemma power set} for $n=3$.}
\label{fig: power set}
\end{figure}

  \section{Proof of Theorem \ref{thm: gravity}}
  \label{sec: gravity}

We begin by showing that the minimal mathematical representation of the 2D gravity ordering, which we denote by $(X,\preceq)$ for simplicity, consists of three functions:

The 2D gravity ordering has a mathematical representation consisting of three functions \eqref{eq: gravity mu}. In order to show that this representation is minimal, we will argue by reduction to the absurd. Let us hence assume that there is a mathematical representation consisting of two function $\{u,v\}$ and consider a couple of points $x,y \in X$ such that $A_x \neq A_y$, where
\begin{equation*}
    A_x \equiv \{ y \in X | \neg \left(x \bowtie y \right)\}.
\end{equation*}
By construction, we can assume that $u(x)<u(y)$ and $v(x)>v(y)$. In fact, we have the following property:

\begin{claim}
\label{claim 3}
If $z \in A_x$ and $t \in A_y$, then $u(z)<u(t)$ and $v(z)>v(t)$.
\end{claim}

\begin{proof}
We note first that we must have $u(x) < u(t)$ and $v(x)>v(t)$ for all $t \in A_y$. To see this, we can assume that $t \prec y$ for some $t \in A_y$ ($t \preceq y$ works analogously). In this scenario, if $u(x) \geq u(t)$, and since $v(x)>v(y) \geq v(t)$, then we get that $t \preceq x$. This contradicts the fact that $t \bowtie x$. We conclude hence that $u(x) < u(t)$. %$v(x)>v(t)$ holds by monotonicity of $v$.

To finish the proof, we show that, given some $z \in A_x$, we must have $u(z) < u(t)$ and $v(z)>v(t)$ for all $t \in A_y$. To see this, we can assume that $x \preceq z$ ($z \prec x$ works analogously). By monotonicity of $v$, we have that $v(z) \geq v(x) > v(t)$. Hence, we must have $u(z) < u(t)$ to avoid contradiction the fact that $z \bowtie t$.
\end{proof}

To conclude, we note that, by definition of mathematical representation, there must exist some $w \in \{ u,v\}$ and
some uncountable family
$(x_i)_{i \in I} \subseteq X$ for which $A_{x_i} \neq A_{x_j}$ provided $i \neq j$ such that, for each $i \in I$, there exists some pair $a_{x_i},b_{x_i} \in A_{x_i}$ fulfilling $a_{x_i} \prec b_{x_i}$ and $w(a_{x_i})<w(b_{x_i})$. This means that we can pick, for each $i \in I$, some $q_{x_i} \in \mathbb Q$ such that $w(a_{x_i})< q_{x_i} <w(b_{x_i})$. Taking $i \neq j$, and assuming $w(x_i) < w(x_j)$ (the opposite case works analogously), we have that $q_{x_i} \neq q_{x_j}$ since
\begin{equation*}
    q_{x_i} < w(b_{x_i}) < w(a_{x_j}) < q_{x_j}
\end{equation*}
 by construction and Claim \ref{claim 3}. As a result, the map
 \begin{alignat*}{3}
    \phi:\text{ } &I &&\to &&\text{ } \mathbb Q\\
    &i &&\mapsto &&\text{ }q_{x_i}
\end{alignat*}
is injective and $\mathbb Q$ is uncountable. Hence, we have reached the desired contradiction and we conclude the minimal mathematical representations of the 2D gravity ordering consist of three functions.

We conclude by showing that the minimal physical representation of the 2D gravity ordering consists of a countably infinite number of functions:

Since the 2D gravity ordering has finite mathematical representations \eqref{eq: gravity mu}, it has physical representations consisting of a countably infinite number of functions \cite{alcantud2016richter}. In order to show that this representation is minimal, we will argue by reduction to the absurd. Let us hence assume that there is a finite physical representation $\mathcal U \equiv (u_i)_{i=1}^n$. The first thing we ought to notice is the following property:

\begin{claim}
\label{claim: separation}
For each pair $x,y \in X$ with $A_x \neq A_y$, there exists some $u_{i_{xy}} \in \mathcal U$ such that
\begin{equation}
\label{eq: separation}
    u_{i_{xy}}(z) < u_{i_{xy}}(t) \text{ } \forall
 z,t \in X \text{ s.t. } z \in A_x \text{ and } t \in A_y.
\end{equation}
\end{claim}

\begin{proof}
We show this by reduction to the absurd. Assume we have some pair $x,y \in X$ such that $A_x \neq A_y$ and, for each $i \in \{1,\dots,n\}$, there exist some $x_i \in A_x$ and $y_i \in A_y$ such that $u_i(x_i) \geq u_i(y_i)$
and define
\begin{equation*}
    \begin{split}
        &x' \equiv \max_{\preceq} \{x_1,\dots,x_n\}, \\
        &y' \equiv \min_{\preceq} \{y_1,\dots,y_n\},
    \end{split}
\end{equation*}
where $\text{max}_{\preceq}$ ($\text{min}_{\preceq}$) refers to the maximum (minimum) in terms of $\preceq$, which is well-defined since $A_x$ ($A_y$) is a chain and we only consider a finite subset of it.
%$x' \equiv \max_{\preceq} \{x_1,\dots,x_n\}$, $z' \equiv \min_{\preceq} \{z_1,\dots,z_n\}$.

By monotonicty of $\mathcal U$, we get that $u_i(y') \leq u_i(x')$ for $i=1,\dots,n$. Hence, we have $y' \preceq x'$. However, $x' \bowtie y'$ since $x' \in A_x$ and $y' \in A_y$ and we have reached the desired contradiction.
\end{proof}

%We first show that the pairs $A_x,A_z$ are separated by one of the functions on $\mathcal U$:

%\begin{claim}
%For each pair $x,z \in \mathbb R$ with $x \neq z$, there exists some $u_{i_{xz}} \in \mathcal U$ such that
%\begin{equation}
%\label{eq: sup < inf notation}
%      \underset{y \in A_x}
%{\text{sup}} \left\{  u_{i_{xz}}(y) \right\} <  \underset{y \in A_z}{\text{inf}} \{  u_{i_{xz}}(y) \}.
%\end{equation}
%\end{claim}

Let us denote by $u_i(A_x) \neq u_i(A_z)$ the fact that $u_i$ \textbf{separates} $A_x$ from $A_y$, that is, $u_i$ fulfills either \eqref{eq: separation} or the version of \eqref{eq: separation} where we exchange $x$ for $y$.
As a consequence of Claim \ref{claim: separation}, there exists some function in $\mathcal U$ that separates all the possible pairs in some uncountable set: 

\begin{claim}
\label{claim: uncountable B_i_0}
There exists some $u_{i_0} \in \mathcal U$ such that $B_{i_0}$ is uncountable, where we define for $i=1,\dots,n$
\begin{equation}
    B_{i} \equiv \underset{B \subseteq X : u_i(A_x) \neq u_i(A_z) \forall x,z \in B}{arg\,max} %\text{argmax}} 
    \left\{ |B| \right\}.
\end{equation}
\end{claim}

\begin{proof}
We show this by reduction to the absurd. Assume hence that $B_i$ is countable for $i=1,\dots,n$ and consider some $x \in X$. By assumption, there must exist some function in $\mathcal U$, say $u_1$, and some uncountable set, say $S_1 \subseteq X$, such that:
\begin{itemize}
    \item $s \bowtie t$ for all $s,t \in S_1$,
    \item $u_1(A_x) \neq u_1(A_s)$ for all $s \in S_1$, and
    \item $u_1(A_s) \neq u_1(A_t)$ does not hold for all $s,t \in S_1$.
\end{itemize}
Since the different elements in $S_1$ are not distinguished by $u_1$, and given $y_1 \in S_1$, there must exist some function in $\mathcal U$, say $u_2$, and some uncountable set $S_2 \subseteq S_1 \smin \{y_1\}$ such that:
\begin{itemize}
    \item $u_2(A_{y_1}) \neq u_2(A_s)$ for all $s \in S_2$, and
    \item $u_2(A_s) \neq u_2(A_t)$ does not hold for all $s,t \in S_2$.
\end{itemize}
We can follow this argument recursively for $i=3,\dots,n$, taking $u_i$ as function and $S_i \subseteq S_{i-1} \smin \{y_{i-1}\}$ for some $y_{i-1} \in S_{i-1}$.

To conclude, and since $S_n$ is uncountable, we can pick two elements $s,t \in S_n$, $s \neq t$. By construction, we have that $s \bowtie t$ and $u_i(A_s) \neq u_i(A_t)$ does not hold for all $i \in \{1,\dots,n\}$. This contradicts Claim \ref{claim: separation} and concludes the proof.
\end{proof}

%Relaying on the separation property of Claim \ref{claim: separation}, we can show that there exists some uncountable $I \subseteq \mathbb R$ such that one of the functions in $\mathcal U$ must separate every pair of sets that we can take from the family $(A_{x_i})_{i \in I}$:

%\begin{claim}
%There exists some $i_0$ such that $B_{i_0}$ is uncountable, where we take
%\begin{equation}
 %   B_{i} \equiv \underset{B \subseteq \mathbb R : u_i(A_x) \neq u_i(A_z) \forall x,z \in B}{\text{argmax}} \left\{ |B| \right\}
%\end{equation}
% for $i=1,\dots,n$ and we denote by $u_i(A_x) \neq u_i(A_z)$ the fact that either plain \eqref{eq: sup < inf notation} or the version of \eqref{eq: sup < inf notation} where we exchange $x$ for $z$ holds.
%\end{claim}

%\begin{proof}
%We show this by reduction to the absurd. Assume $B_i$ is countable for $i=1,\dots,n$. Since there is an uncountable number of pairs $x,y \in \mathbb R$ that we need to separate through some $u_i$ as in Claim \ref{claim: separation}, and $\mathcal U$ is finite, there must be some function, say $u_{1}$, that separates and uncountable number of pairs. Since $B_{1}$ is countable, 

%HEREEEEEEEEEEEEEEEEEE

%\end{proof}

To conclude, let us consider $B_{i_0}$ from Claim \ref{claim: uncountable B_i_0} and note that, for each $x \in B_{i_0}$, we can pick some $y_x \in A_x$ and some $q_x \in \mathbb Q$ such that $u_{i_0}(y_x) < q_x <u_{i_0}(x)$. By construction of $B_{i_0}$, we have that $q_x \neq q_z$ provided $x,z \in B_{i_0}$ and $x \neq z$. Hence, the map
 \begin{alignat*}{3}
    \phi:\text{ } &B_{i_0} &&\to &&\text{ } \mathbb Q\\
    &x &&\mapsto &&\text{ }q_{x}
\end{alignat*}
is injective, $\mathbb Q$ is uncountable and we have reached a contradiction. This concludes the proof.

\begin{rem}
If we consider the same ordering, buy we only allow a countable set of horizontal components, then Theorem \ref{thm: gravity} does not hold. In fact, we end up with an ordering that has minimal mathematical and physical representations consisting of two functions.
\end{rem}

%We conclude by taking $B_{i_0}$ according to Claim .. and by noting that the ordering associated to $u_{i_0}$,
%\begin{equation*}
%(x,y) \preceq_{u_{i_0}} (z,t)  \iff u_{i_0}((x,y)) \leq u_{i_0}((z,t)), 
%\end{equation*}
%cannot be Debreu separable since, for any Debreu dense subset $D \subseteq \mathbb R^2$, there must be some $d_x \in A_x \bigcap D$ for each $x \in B_{i_0}$. Now, since $B_{i_0}$ is uncountable by Claim \ref{claim: uncountable B_i_0} and $d_x \neq d_y$ for all $x,y \in B_{i_0}$ with $x \neq y$, $D$ is uncountable and $\preceq_{u_{i_0}}$ is not Debreu separable. This contradicts \cite[Theorem 1.4.8]{bridges2013representations} and concludes the proof.

\section{Proof of Theorem \ref{finite geo infinite deb}}
\label{sec: proof main theo}

We consider the partial order $(X,\preceq)$, where we take the set $X \equiv \mathbb{R}\setminus \{0\}$ equipped with the partial order $\preceq$ defined by
%Take $(X,\preceq)$ where $X \coloneqq \mathbb{R}/\{0\}$ equipped with $\preceq$, where 
\begin{equation}
\label{eq: ex thm 1}
    x \preceq y \iff
    \begin{cases}
    |x| \leq |y| \text{, and } \\ \sgn(x) \leq \sgn(y)
    \end{cases}
\end{equation}
for all $ x,y \in X$, where $|\cdot|$ is the absolute value function and $\sgn(\cdot)$ is the sign function, that is,
  \begin{equation*}
    \sgn(x) \equiv
    \begin{cases}
    1 \text{ if } x \geq 0,\\
    0 \text{ otherwise.}
    \end{cases}
\end{equation*}
%$sgn(x) \coloneqq 1$ if $x> 0$ and $sgn(x) \coloneqq -1$ if $x<0$.
A representation of $(X,\preceq)$ can be found in Figure \ref{fig:counterex}.

 \begin{figure}[!tb]
\centering
\begin{tikzpicture}
    \node[main node] at (0,-2) (1) {$-x$};
    \node[main node] at (-2,0) (2)  {$x$};
    \node[main node] at (2,0) (3)  {$-y$};
     \node[main node] at (-2,2) (4)  {$y$};
    \node[main node] at (2,2) (5)  {$-z$};
    \node[main node] at (0,4) (6) {$z$};

    \path[draw,thick,-]
    (1) edge node {} (2)
    (1) edge node {} (3)
    (2) edge node {} (4)
    (3) edge node {} (4)
    (3) edge node {} (5)
    (4) edge node {} (6)
    (5) edge node {} (6)
    ;
\end{tikzpicture}
\caption{A Hasse diagram of the partial order constructed in Theorem \ref{finite geo infinite deb}. We assume here that $0<x<y<z$.}
\label{fig:counterex}
\end{figure}

We begin by noting that the right hand side of \eqref{eq: ex thm 1} is a minimal mathematical representation, since the partial order is not total. We also know \cite{alcantud2016richter} that the minimal physical representation of $(X,\preceq)$ is at most countably infinite. 

To complete the proof, we show by reduction to the absurd that no finite physical presentations exist, which means the minimal physical representation of $(X,\preceq)$ is countably infinite. Let us assume, hence, that there exists some $N< \infty$ such that $(v_i)_{i=1}^N$ is a physical representation. If this is the case, then the following should hold:

\begin{claim}
\label{claim I}
 For each $0<x \in X$ there exists some $i_x$, $1\leq i_x\leq N$, such that $v_{i_x}(x) \leq v_{i_x}(-y)$ for all $x<y \in X$.
\end{claim}

\begin{proof}
 We show this by reduction to the absurd. To do so, fix some $0<x \in X$ for which we assume the result does not hold and define, for $1\leq i\leq N$,
\begin{equation*}
\begin{split}
    &A_i  \equiv \{y \in X| x< y \text{ and } v_i(x) \leq v_i(-y) \},\\
    &x_i \equiv
    \begin{cases}
    \inf A_i \text{ if } A_i \neq \emptyset,\\
    \infty \text{ otherwise,}
    \end{cases}
\end{split}
\end{equation*}
where \textbf{inf} denotes the infimum in the $\leq$ sense.

By hypothesis, we have that $x < x_i$ for $1 \leq i \leq N$ since, otherwise, we would reach a contradiction and the proof would be finished.
Thus, we can pick some $z \in X$ such that 
\begin{equation*}
x < z < \min \{x_1,..,x_N\}.    
\end{equation*}
Thus, we have that $v_{i}(-z) < v_{i}(x)$ for $1 \leq i \leq N$ and, by definition of physical reresentation, we get that $-z \prec x$. This contradicts the fact that $-z \bowtie x$ (which holds since $x < z$) and concludes the proof.
\end{proof}

Since $N$ is finite, there must exist some $i_0$, $1 \leq i_0 \leq N$, for which Claim \ref{claim I} holds for an uncountable subset of strictly positive elements $X_0 \subseteq X$. Since $v_{i_0}$ is a strict monotone, we can pick, for each $x \in X_0$, some $q_x \in \mathbb Q$ such that $v_{i_0}(-x) < q_x < v_{i_0}(x)$. Now, if we consider $x,y \in X_0$ such that $x < y$, then we have that $q_x \neq q_y$ since
\begin{equation*}
    q_x < v_{i_0}(x) \leq v_{i_0}(-y) < q_y
\end{equation*}
 by construction. As a result, the map
 \begin{alignat*}{3}
    \phi:\text{ } &X_0 &&\to &&\text{ } \mathbb Q\\
    &x &&\mapsto &&\text{ }q_x
\end{alignat*}
is injective and $\mathbb Q$ is uncountable. Hence, we have reached the desired contradiction that concludes the proof.

\section{Proof of Theorem \ref{dim majo}}
\label{sec: proof major thm}

The statement regarding the minimal physical representation is shown in the accompanying paper \cite[Theorem 1]{hack2022disorder}. The one regarding the minimal mathematical representations is stated in the following lemma:

\begin{lemma}
\label{geo d-m dim majo}
If $|\Omega| \geq 2$, then $(\mathbb P_\Omega^{\downarrow},\preceq_m)$ has a minimal mathematical representation consisting of $|\Omega|-1$ functions. 
\end{lemma}

We devote the rest of this section to proving Lemma \ref{geo d-m dim majo}. Let us first consider the simplest cases:
\begin{itemize}
    \item If $|\Omega|=2$, the result follows by definition \eqref{majo}.
    \item  If $|\Omega|= 3$, then $(\mathbb P_\Omega^{\downarrow},\preceq_m)$ has a mathematical representation consisting of two functions by definition \eqref{majo} and it does not have a representation with a single function since it is not total.
\end{itemize}

 If $|\Omega| \geq 4$, then $(\mathbb P_\Omega^{\downarrow},\preceq_m)$ has a mathematical representation consisting of $|\Omega|-1$ functions by definition \eqref{majo}.
 To conclude, we ought to show that $(\mathbb P_\Omega^{\downarrow},\preceq_m)$ has no mathematical representation consisting of $|\Omega|-2$ functions. 
 In order to do so, it suffices to construct a subset $S \subset P_\Omega^{\downarrow}$ that is order isomorphic to the subset
\begin{equation*}
S_0 \equiv \{m \in \mathbb Z | m \text{ or } -m \in \{1,\dots,|\Omega|-1\} \}
\end{equation*}
 of the partial order in \cite[Proposition 7]{hack2022classification}. We can then follow \cite[Proposition 7]{hack2022classification} to conclude that no mathematical representation consisting of $|\Omega|-2$ functions exists.

%For simplicity, we conclude indicating the general idea that needs to be followed in order to construct $S$. We begin taking some distribution $p^1 \in P_\Omega^{\downarrow}$ such that $p^1_i>p^1_{i+1}>0$ for all $i$ such that $1 \leq i <|\Omega|$. $p^1$ is then used to sequentially define all the distributions in $S \coloneqq (p_i,q_i)_{i=1}^{|\Omega|-1}$, where each distribution is defined from the previous one by simple adding or subtracting some probability mass at the appropriate component in order to fulfill the following relations:
For simplicity, we conclude by sketching the construction of $S$:

We begin taking some distribution $p_1 \in P_\Omega^{\downarrow}$ such that $(p_1)_i>(p_1)_{i+1}>0$ for all $i$ such that $1 \leq i <|\Omega|$. $p_1$ is then used to sequentially define the set of distributions 
\begin{equation*}
    S \equiv (p_i,q_i)_{i=1}^{|\Omega|-1},
\end{equation*}
 where each distribution is defined from the previous one by simple adding or subtracting some probability mass at the appropriate component in order to fulfill the following relations:
\begin{widetext}
\begin{equation}
\label{eq: order iso}
\begin{split}
   u_1(p_1) &> \dots >u_1(p_{n-1})>u_1(q_n)>u_1(p_n)>u_1(q_1)>\dots> u_1(q_{n-1}), \\
  u_2(p_n) &> u_2(p_1)>\dots >u_2(p_{n-2})>u_2(q_{n-1})>u_2(p_{n-1})>u_2(q_n)>\\
  &>u_2(q_1)>\dots >u_2(q_{n-2}), \\
   &\vdots\\
   u_n(p_2) &> \dots >u_n(p_{n})>u_n(q_1)>u_n(p_1)>u_n(q_2)>\dots> u_n(q_{n}),
\end{split}
\end{equation}
\end{widetext}
where each restriction is obtained from the previous one by permuting the subindexes and we use $n \equiv|\Omega|-1$ for commodity.
It is direct to check the order isomorphism between $S$ and $S_0$ from \eqref{eq: order iso}. This concludes the proof.

\section{Proof of Theorem \ref{teo: inf majo}}
\label{app: infinite majo}

We can prove the statement regarding mathematical representations by reduction to the absurd. By definition \eqref{def: inf majo}, we have a countably infinite mathematical representation. Let us assume there exists a finite mathematical representation $(u_i)_{i=1}^k$ for $(\mathbb P_\infty^{\downarrow},\preceq_m^{\infty})$, and note that there is a natural order isomorphism between the subset
\begin{equation*}
    \{p \in \mathbb P_\infty^{\downarrow} | p_{n} = 0 \text{ for } n \geq k+3\}
\end{equation*}
with the induced ordering and $(\mathbb P_{\Omega}^{\downarrow},\preceq_m)$ with $|\Omega|=k+2$. Hence, $(u_i)_{i=1}^k$ is a mathematical representation for $(\mathbb P_{\Omega}^{\downarrow},\preceq_m)$ as well. However, the minimal mathematical representation of $(\mathbb P_{\Omega}^{\downarrow},\preceq_m)$ consists of $k+1$ functions by Lemma \ref{geo d-m dim majo}. Thus, $(u_i)_{i=1}^k$ is not a mathematical representation of infinite majorization.

Since we have a countably infinite mathematical representation, we also have a countably infinite physical representation by \cite{alcantud2016richter}.

\section{Debreu separability and finite representations}
\label{sec: debreu thm}

\subsection{Proof of Theorem \ref{classification}}

Since the converse holds by definition and without requiring Debreu separability on $(X,\preceq)$, we simply show that a finite physical representation $(v_i)_{i = 1}^k$ can be constructed from a finite multi-utility $(u_i)_{i = 1}^k$ provided $(X,\preceq)$ is Debreu separable. In order to do so, in analogy with \cite[Proposition 10]{hack2022classification}, it suffices to show that
\begin{equation*}
    I_{i} \equiv \{r \in \mathbb{R}| \exists x,y \in X \text{ such that } x,y \in u_i^{-1}(r) \text{ and } x \prec y\}
\end{equation*}
 is countable for $i=1,\dots,k$. Fix thus some $i \in \{1,\dots,k\}$ and a countable Debreu dense subset $D \subseteq X$ and note that, for each $r \in I_{i}$, there exists some $d_r \in D$ such that $d_r \in u_i^{-1}(r)$. This is the case since, by definition, there exist $x,y \in X$ such that $x \prec y$ and $x,y \in u_i^{-1}(r)$. We can pick as $d_r$ any $d \in D$ such that $x \preceq d \preceq y$ since
 \begin{equation*}
 r=u_i(x) \leq u_i(d) \leq u_i(y)=r    
 \end{equation*}
 by monotonicity of $u_i$. To conclude, note that the map
 \begin{alignat*}{3}
    \phi:\text{ } &I_i &&\to &&\text{ } D\\
    &r &&\mapsto &&\text{ }d_r
\end{alignat*}
is injective. This implies that $I_i$ is countable and concludes the proof.
 
 %As a result, for all $r \in I_{u_i}$, we have that $D \bigcap u_i^{-1}(r) \neq \emptyset$ and, hence, we can fix some $d_r \in D \bigcap u_i^{-1}(r)$. To conclude, we simply notice the map $f: I_{u_i} \to D$, $r \mapsto d_r$ is injective.

\subsection{The converse}

As we show in Proposition \ref{finite RP no deb sep}, while Debreu separability makes mathematical and physical representations equivalent, it is not implied by the stronger of them. The same happens for the closely related notion of Debreu upper separability: We say a partial order $(X,\preceq)$ is \textbf{Debreu upper separable} provided there exists some countable \textbf{Debreu upper dense} subset $D \subseteq X$, that is, provided there exists some $d \in D$ fulfilling $x 
\bowtie d \preceq y$ for each pair $x,y \in X$ such that $x \bowtie y$.

\begin{prop}
\label{finite RP no deb sep}
There exist partial orders with finite physical representations
such that every subset that is either Debreu dense or Debreu upper dense is uncountable.
\end{prop}

\begin{proof}
Consider the partial order $(X,\preceq)$, that consists of $X \coloneqq \mathbb R \setminus \{0\}$ equipped with $\preceq$, where 
\begin{equation*}
    x \preceq y \iff u_1(x) \leq u_1(y) \text{ and } u_2(x) \leq u_2(y),
\end{equation*}
with $u_1(x) \equiv x$ and $u_2(x) \equiv 1/x$ for all $x \in X$. See Figure \ref{fig 2} for a representation of $(X,\preceq)$,
%$u_2(x) \equiv 1/x$ if $x >0$ and $u_2(x) \equiv -1/|x|$if $x <0$.
and note that $\{u_1,u_2\}$ is a finite strict monotone multi-utility.
%and that it is direct to check that they are lower semicontinuous in $\sigma(P)$ since they are monotone and, for every directed set $D \subseteq X$, we have that $\sqcup D \in D$.

Regarding Debreu density, take a Debreu dense subset $D \subseteq X$.
%and define for all $x \in [0,1]$ $z_x \coloneqq x+2$.
Since $-x \prec x$ holds for all $x >0$, there exists some $d_x$ such that $-x \preceq d_x \preceq x$. By construction, this means that $d_x \in \{-x,x\}$ and, hence, $d_x \neq d_y$. This implies that $D$ is uncountable. Regarding Debreu upper density, take a Debreu upper dense subset $D \subseteq X$. Since the $\preceq$-minimal elements are incomparable, they must all belong to $D$. Since the minimal elements are uncountable, $D$ is uncountable.
%$x \bowtie y$ for all $x,y>0$ such that $x \neq y$, and $z \prec x,y$ for all $z <0$ we have that the only element $t \in X$ that fulfills $y \bowtie t \preceq x$ is $t \equiv x$. This implies that $D$ is uncountable. 
\end{proof}

\begin{figure}[!tb]
\centering
\begin{tikzpicture}
%\node[rounded corners, draw,fill=blue!10, text height = 3cm, minimum width = 11cm,xshift=4cm,label={[anchor=west,left=.1cm]180:\textbf{B}}] {};
%\node[rounded corners, draw,fill=blue!10, text height = 3cm, minimum width = 11cm,xshift=4cm,yshift=-4cm,label={[anchor=west,left=.1cm]180:\textbf{A}}] {};
    \node[main node] (1) {$x$};
    \node[main node] (2) [right = 2cm  of 1]  {$y$};
    \node[main node] (3) [right = 2cm  of 2]  {$z$};
    \node[main node] (4) [below = 2cm  of 1] {$-x$};
    \node[main node] (5) [right = 2cm  of 4] {$-y$};
    \node[main node] (6) [right = 2cm  of 5] {$-z$};

    \path[draw,thick,-]
    (4) edge node {} (1)
    (4) edge node {} (2)
    (4) edge node {} (3)
    (5) edge node {} (3)
    (5) edge node {} (2)
    (5) edge node {} (1)
    (6) edge node {} (1)
    (6) edge node {} (2)
    (6) edge node {} (3)
    ;
\end{tikzpicture}
\caption{A Hasse diagram of the partial order constructed in Proposition \ref{finite RP no deb sep}. We assume here that $0<x<y<z$.}
\label{fig 2}
\end{figure}

\subsection{Injective physical representations}

A preliminary results regarding the connection between physical representations and injective physical representations is the following:

\begin{itemize}
    \item Any physical representation consisting of two functions is an injective physical representation.
\end{itemize}

We can show this by reduction to the absurd. If we have some partial order $(X,\preceq)$ and some physical representation $\{u,v\}$ that is not injective, then there must exist a pair $x,y \in X$ and a function in the representation, say $u$, such that $x \bowtie y$ and $u(x)=u(y)$. However, if this is the case, and assuming w.l.o.g. that $v(x) \leq v(y)$, the we must have $x \preceq y$, which contradicts the fact that $x \bowtie y$.

This statement together with Theorem \ref{classification} gives us that, if $(X,\preceq)$ is a Debreu separable partial order, then it has a mathematical representation consisting of two functions if and only if it has a physical representation consisting of two injective functions.

\section{Proof of Theorem \ref{teo: dm and geo}}
\label{sec: dm and geo}

To show the result, we consider the \textbf{lexicographic plane} $(\mathbb{R}^2,\preceq_L)$ \cite{debreu1954representation}, where we have
 \begin{equation*}
(x,y) \preceq_L (z,t)  \iff
\begin{cases}
    x < z\\
    x=z  \text{ and } y \leq t
    \end{cases}   
\end{equation*}
for all $ (x,y), (z,t) \in \mathbb{R}^2$.
  Since it is a total order, its Dushnik-Miller dimension is one. However, it has no \textbf{utility function} \cite[Example 1.4.1]{bridges2013representations}, i.e. multi-utility consisting of a single function. Hence, by \cite[Proposition 4.1]{alcantud2016richter}, it cannot have a countable multi-utility and its geometrical dimension is uncountable.

\end{appendix}
\end{document}